\documentclass{article}
\usepackage{amsmath,amsfonts,amssymb,amsthm,graphicx,color}
\usepackage{lineno}
\usepackage{enumerate}

\newtheorem{result}{Main result}
\newcommand{\defi}[1]{\emph{#1}}
\newcommand{\TODO}[1]{}

\numberwithin{equation}{section}

\newtheorem{teo}{Theorem}[section]

\newtheorem{theorem}[teo]{Theorem}

\newtheorem{thm}[teo]{Theorem}
\newtheorem{lemma}[teo]{Lemma}
\newtheorem{cor}[teo]{Corollary}

\newtheorem{prop}[teo]{Proposition}

\theoremstyle{definition}

\newtheorem{remark}[teo]{Remark}

\newcommand{\diam}[1]{\operatorname{diam}({#1})}
\DeclareMathOperator{\freq}{freq}
\newcommand{\vol}[1]{\operatorname{vol}({#1})}

\def\R{\mathcal{R}}

\newcommand{\norm}[1]{\Vert#1\Vert}
\newcommand{\abs}[1]{\lvert#1\rvert}
\newcommand{\interior}[1]{\mathring{#1}}
\newcommand{\leb}[1]{\operatorname{vol}({#1})}

\newcommand{\Aa}{{\mathcal A}}
\newcommand{\B}{{\mathcal B}}

\newcommand{\D}{{\mathcal D}}

\newcommand{\adh}[1]{\overline{#1}}

\newcommand{\NN}{{\bf N}}

\newcommand{\RR}{\mathbb{R}}

\def\NN{\mathbb{N}}

\def\RR{\mathbb{R}}

\def\A{\mathcal{A}}
\def\B{\mathcal{B}}

\def\D{\mathcal{D}}

\def\V{\mathcal{V}}

\def\P{\mathcal{P}}

\def\R{\mathcal{R}}
\def\S{\mathcal{S}}
\def\T{\mathcal{T}}
\newcommand{\pp}{\ensuremath{\mathbf{p}}}

\def\b {\beta}

\begin{document}
\title{Tower systems for Linearly repetitive Delone sets}

\author{JOS\'E  ALISTE-PRIETO and DANIEL CORONEL}
    
    


\maketitle
\begin{abstract}
In this paper we study linearly repetitive Delone sets and prove, following the work of Bellissard, Benedetti and Gambaudo, that the hull of a linearly repetitive Delone set admits a properly nested sequence of box decompositions (tower system) with strictly positive and uniformly bounded (in size and norm) transition matrices. This generalizes a result of Durand for linearly recurrent symbolic systems. Furthermore, we apply this result to give a new proof of a classic estimation of Lagarias and Pleasants on the rate of convergence of patch-frequencies. 
\end{abstract}

\section{Introduction}  
Delone sets arise naturally as mathematical models for the description of solids. In this modelization, the solid is supposed to be infinitely extended and its atoms are represented by points. These atoms interact through a potential (for example a Lennard-Jones potential). For a given specific energy, Delone sets are good candidates to describe the ground state configuration: uniform discreteness corresponds to the existence of a minimum distance between atoms due to the repulsion forces between nuclei, and relative density corresponds to the fact that empty regions can not be arbitrarily big because of the contraction forces. In perfect crystals, atoms are ordered in a repeating pattern extending in all three-dimensions and can be modeled by lattices in $\RR^3$. Quasi-crystalline solids are those whose X-ray diffraction image have sharp spots indicating \emph{long-range order} but without having a full-lattice of periods. Typically, they exhibit symmetries that are impossible for a perfect crystal (see e.g. \cite{shechtman}).

From the mathematical and physical point of view, linear repetitivity (introduced by Lagarias and Pleasants in \cite{LP}) has become a key feature (see e.g. \cite{Sol2,DaLe}), and many known examples of quasi-crystalline solids may be modeled using linearly repetitive Delone sets.

A standard tool in the study of a Delone set $X$ is provided by its hull $\Omega$ and the natural $\RR^d$-action on it. The hull is defined as an appropriate closure of the family $\{X - v : v\in\RR^d\}$ and $\RR^d$ acts over the hull by translation. From this point of view, the hull can be regarded as a  generalization to higher dimensions of the standard orbit-closure construction for aperiodic sequences in symbolic dynamics (see e.g. \cite{Robi}). In this context, linear repetitivity corresponds to linear recurrence (see \cite{Durand}).

A powerful combinatorial tool in the study of linearly recurrent subshifts is provided by  Kakutani-Rohlin towers (see e.g. \cite{HPS,CDHM}). In particular, a fundamental result in this context is the following (see \cite{Durand,CDHM}): 
\begin{teo}
\label{teoCDHM}
Let $(X,T)$ be an aperiodic linearly recurrent subshift. Then there exist a sequence of Kakutani-Rohlin (KR) partitions and $M>0$ 
such that 
\begin{enumerate}[(i)]
\item the number of KR towers of level $n$ is uniformly bounded by $M$,
\item every KR tower of level $n$ crosses all the KR towers of level $n-1$ and the total number of these crossings is uniformly bounded by $M$. 
\end{enumerate}
\end{teo}
An equivalent statement for Theorem \ref{teoCDHM} is that linearly recurrent subshifts admit nested sequences of KR partitions with transition matrices that are strictly positive and uniformly bounded in norm. As a corollary, for instance, it is easy to deduce that linearly recurrent subshifts are uniquely ergodic (see e.g. \cite{Durand,CDHM}). Other applications include the study of spectral properties of these systems (see \cite{CDHM,BDM,BDM2}).
  
In the context of Delone sets, tower systems were introduced in \cite{BBG}, and provide a proper generalization for sequence of KR towers (see Section \ref{sec:towers} for the definition). A natural question  we answer in this paper is whether an analog of Theorem \ref{teoCDHM} for linearly repetitive Delone systems and tower systems holds. To answer this question, we refine the construction in \cite{BBG} to linearly repetitive systems to obtain several estimations on parameters that control the growth of the towers and, therefore, also the norms of the associated transition matrices. Our main result is the following:   
\begin{result}[\emph{cf.} Theorem \ref{thm:towersLR2}]
Let $\Omega$ be the hull of a linearly repetitive Delone set $X$. Then there exist a tower system  and $M>0$  such that 
\begin{enumerate}[(i)]
\item  the number of boxes at each level is uniformly bounded by $M$,
\item  every box of level $n$ crosses every box of level $n-1$ and the total number of crossings is uniformly bounded by  $M$. 
\end{enumerate}
\end{result}

For similar constructions, we refer to \cite{priebe1,PS,LeSt,Besbes,GMPS,Forrest, CGM}.

In \cite{BG,GaMar}, it is proved (using standard arguments) that Delone systems satisfying the assertions of the main result are uniquely ergodic. Hence, as a corollary of this result and our main result, we obtain an alternative proof of the unique ergodicity of the hull of an aperiodic linearly repetitive Delone set $X$. This is equivalent to say (see e.g. \cite{LMS}) that $X$ has uniform patch frequencies, i.e., for every patch 
$\pp$ in $X$, the number $n_{\pp}(U)$ of patches of $X$ that are equivalent to $\pp$ and whose center is included in a $d$-cube $U$ of 
side $N$, divided by the volume of $U$ converges to a limit, called the \emph{frequency} of $\pp$ and denoted by $\freq(\pp)$, when $N$ goes to infinity.
The existence of uniform patch frequencies is well-known for linearly repetitive Delone sets, and it is a consequence of the following 
stronger result of Lagarias and Plesants \cite{LP}:
\begin{teo}[Lagarias-Pleasants' theorem]
\label{teo:LPintro}
Let $X$ be a linearly repetitive Delone set. Then, $X$ has uniform patch frequencies. Moreover, there exists $\delta>0$ such that,
for every patch $\pp$ of $X$,
\begin{equation*}
\left|\frac{n_\pp(D_N)}{\vol{D_N}} - \operatorname{freq}(\pp)\right| = O(N^{-\delta}),
\end{equation*}
 where $D_N$ is either a $d$-cube with side $N$ or a ball  of radius $N$.
\end{teo}

In this paper, we give an alternative proof of Lagarias and Pleasants' theorem as an application of our main result. A key step in the proof
is the introduction of a Markov chain associated with the tower system, whose mixing rate is related to the constant $\delta$. We remark that in the original proof of Lagarias and Pleasants, the constant $\delta$ depends on the geometry of $D_N$. Our new approach suggests that the $\delta$ should be defined purely in terms of $X$ and the tower system, and so we expect that the proof can be extended to provide estimations for more general  additive ergodic theorems in \cite{DaLe,Besbesb,LeSt}. We also remark that the proof can be applied to self-similar systems, with better bounds, and this will be the subject of a forthcoming paper. 

We finish the introduction by giving the organization of the paper. In Section \ref{sec:back} we review the theory on Delone set and the dynamical system approach that will be needed in the sequel. In Section \ref{sec:boxes}, we review the definition of tower system and prove the main result, by supposing the existence of a tower system satisfying some extra conditions (see Theorem \ref{thm:towersLR}). In Section \ref{proof:gambaudo} we prove the existence of this tower system. Then, the construction of the Markov chain and the proof of Lagarias-Pleasant's theorem are given, respectively, in  Section \ref{sec:markov} and  Section \ref{secLP}.
\section{Background}
\label{sec:back}
In this section we fix some notation and review some basic definitions about Delone sets and the associated hulls. For details we refer to \cite{LP,Robi}. For some of the notation, we also follow \cite{LeSt}. 

We work with the $d$-dimensional Euclidean space, denoted by $\RR^d$. The set of non-negative integer numbers will be denoted by $\NN$, and the set of positive integer numbers by $\NN^*$. If $\P$ is a subset of $\RR^d$ and $v$ is in $\RR^d$, then the set $\{ P - v : P\in \P\}$ will be denoted $\P-v$. A pair $(\Lambda,Q)$, where $Q$ is a bounded subset of $\RR^d$ and $\Lambda\subseteq Q$ is finite is called a \emph{pattern}. If $Q=\adh{B}_S(x_0)$ is the closed ball of radius $S>0$ around $x_0\in \Lambda$, then $(\Lambda,Q)$ is called a \emph{$S$-pattern}, \emph{centered} at $x_0$. The set $Q$ is called the \emph{support} of the pattern. For $t\in\RR^d$ and $(\Lambda,Q)$, we set $(\Lambda,Q)-t = (\Lambda-t,Q-t)$. Two patterns $(\Lambda_1,Q_1)$ and $(\Lambda_2,Q_2)$ are \emph{equivalent} if there exists $t\in\RR^d$ such that $(\Lambda_1,Q_1)-t = (\Lambda_2,Q_2)$. We refer to the equivalence class of a pattern $(\Lambda,Q)$
as a pattern-class.
 
\subsection{Delone sets.}

Let $X$ be a subset of $\RR^d$. We say that $X$ is $r$-discrete if every closed ball of radius $r$ in $\RR^d$ intersects $X$ in at most one point. We say that $X$ is $R$-dense if every closed ball of radius $R$ in $\RR^d$ intersects $X$ in at least one point. The set $X$  is a \emph{Delone set} if it is  $r$-discrete and $R$-dense for some $r,R>0$.

Let $X$ be a Delone set. A \emph{patch} in $X$ is a pattern of the form $X\wedge Q:= (X\cap Q, Q)$. The \defi{$S$-patch} of $X$ \emph{centered} at $x\in X$ is defined as $X\wedge \adh{B}_{S}(x)$. 
Two patches in $X$ are \emph{equivalent} if they are equivalent as patterns. The \emph{set of pattern-classes} of $X$ is the set containing the pattern-classes of all patches in $X$. 

For a $S$-pattern $\pp = (\Lambda,\adh{B}_{S}(x_0))$, an \emph{occurrence} of $\pp$ in $X$ is a point $y\in\RR^d$ such that $X\wedge \adh{B}_S(y)$ is equivalent to $\pp$. A Delone set $X$ is \emph{repetitive} if for every $S>0$ there is $M>0$ such that every ball of radius $M$ contains an occurrence of every $S$-patch of $X$. The smallest such $M$ is denoted by $M_X(S)$. If there exists $L>1$ such that $M_X(S)\leq LS$ for all $S>0$, then $X$ is called \emph{linearly repetitive}.  
The set $X$ has \emph{finite local complexity} if the number of equivalence classes of $S$-patches is finite. Clearly, every repetitive Delone set has finite local complexity but not the converse. A Delone set $X$ is called \emph{aperiodic} if $X - v \neq  X$ for all $v \neq 0$.

The collection of all Delone sets with finite local complexity is denoted by $\D$. Given two Delone sets $X$ and $Y$ in $\D$, their distance is defined as the smallest $0<\varepsilon<\sqrt{2}/2$ for which there exist $u,u'\in B_{\varepsilon}(0)$ such that 
\begin{equation*}
(X-u)\cap \adh{B}_{1/\varepsilon}(0) = (Y-u')\cap \adh{B}_{1/\varepsilon}(0);
\end{equation*}
if such $\varepsilon$ does not exist, then the distance between $X$ and $Y$ is defined to be $\sqrt{2}/2$. It is standard that this gives  a distance (see e.g. \cite{LMS}) under which two Delone sets are close whenever they coincide in a big ball around $0$ up to a small translation. We  refer to the topology
induced by this metric as the \emph{tiling} topology (for a discussion of different topologies for Delone sets, see \cite{Moody}).

Given a Delone set $X$ in $\D$, the hull of $X$, denoted by $\Omega$, is defined as the closure w.r.t. the tiling topology of the family 
$\{X - v : v\in\RR^d\}$. It is well-known that $\Omega$ is compact, and the set of pattern-classes of $X$ contains the pattern-classes of all Delone sets in $\Omega$. It is important to observe that this implies that if $X$ is $r$-dense and $R$-discrete for some fixed $r$ and $R>0$, then all the Delone sets in $\Omega$ are also $r$-dense and $R$-discrete (for the same choices of $r$ and $R$).
 
The \emph{translation action} $\Gamma$ over $\Omega$ is the $\RR^d$-action  defined by 
\[\Gamma_{v} Y = Y - v\]
for all $v\in\RR^d$ and $Y\in \Omega$. The pair $(\Omega,\Gamma)$ forms a dynamical system and we refer to it as a  \emph{Delone system}. Recall that $(\Omega,\Gamma)$ is said to be \emph{minimal} if every orbit is dense. It is well-known that $\Omega$ is minimal iff $X$ is repetitive, and in this case, $\Omega$ is the set containing all the Delone sets that have the same pattern-classes as $X$. Moreover, if $X$ is repetitive and aperiodic, then 
all Delone sets in $\Omega$ are also aperiodic.

\subsection{Local transversals and return vectors.}
Let $(\Omega,\Gamma)$ be an aperiodic  minimal Delone system. The \emph{canonical transversal} of $\Omega$ is the set composed of  all 
Delone sets in $\Omega$ that contain $0$. This terminology is motivated by the fact that if $Y$ is in $\Omega^0$, then every small translation of $Y$
will not be in $\Omega^0$. A \emph{cylinder} in $\Omega$ is a set of the form 
\begin{equation*}
C_{Y,S}:=\{Z\in\Omega\mid Z\wedge \adh{B}_S(0) = Y\wedge \adh{B}_S(0)\},
\end{equation*}
where $Y\in\Omega$ and $S>0$ are such that $Y\cap \adh{B}_S(0)\neq\emptyset$. The following proposition is well-known (see e.g. \cite{KP}). 

\begin{prop}
\label{prop:cantor}
Every cylinder in $\Omega$ is a Cantor set. Moreover, a basis for the topology of $\Omega$ is given by sets of the form 
\[\{ Z - v \mid  Z\in C_{Y,S}, v\in B_{\varepsilon}(0)\}.\]
In particular, the canonical transversal $\Omega^0$ is a Cantor set.
\end{prop} 

A \emph{local transversal} in $\Omega$ is a clopen (both closed and open) subset of any cylinder in $\Omega$. By Proposition \ref{prop:cantor}, a local transversal $C$ is a Cantor set. This implies that 
\begin{equation*}
\operatorname{rec}(C) := \inf\{S>0 \mid C_{Y,S}\subseteq C \text{ for all } Y\in C\}
\end{equation*}
is finite, and the collection 
\begin{equation*}
\{ C_{Y,S} \mid Y\in C, S > \operatorname{rec}(C)\}
\end{equation*}
forms a basis for its topology. Indeed, since $C$ is a Cantor set, it is easy to find a finite set $\{Y_1,\ldots,Y_m\}$ in $C$ such that 
\begin{equation*}
 C  = \bigcup_{i=1}^m C_{Y_i,\operatorname{rec}(C)}.
\end{equation*}
The motivation to define $\operatorname{rec}(C)$ is the following: suppose that we are given a Delone set $Y\in \Omega$ and we want to check if $Y$ belongs to $C$. Then it suffices to look whether the patch $Y\wedge \adh{B}_{\operatorname{rec}(Y)}(0)$ is equivalent to $Y_i\wedge \adh{B}_{\operatorname{rec}(Y)}(0)$ for some $Y_i$. Of course, if $C = C_{Y,S}$, then its recognition radius is smaller than $S$.

Given a  local transversal $C$ and $D\subseteq\RR^d$,  the following notation will be used throughout the paper:
\begin{equation*}
C[D] = \{ Y - x \mid Y\in C, x\in D\}.
\end{equation*}
A very successful way of studying the hull is provided by the set of return vectors to a local transversal. Given a local transversal $C$  and a Delone set $Y\in\Omega$, we define 
\begin{equation*}
\R_C(Y) = \{x\in \RR^d \mid  Y -x \in C\}.
\end{equation*}
When $Y$ belongs to $C$, we refer to $\R_C(Y)$ as the \defi{set of return vectors} of $Y$ to $C$. 
The following lemma is standard (see e.g.\cite{Coronel1})
\begin{lemma}
Let $C$ be a local transversal. Then for each $Y \in C$, the set of return vectors $\R_C(Y)$  is a repetitive Delone set. Moreover, 
the following quantities
\begin{align}
\label{rX}
r(C) = &\frac{1}{2}\inf\{\norm{x-y}\mid x,y \in \R_C(Y),\, x\neq y\},\quad \text{ and} \\
\label{RX}
R(C) =&\inf\{R>0 \mid \R_C(Y)\cap \adh{B}_R(y)\neq\emptyset \text{ for all $y\in\RR^d$} \},
\end{align}
do not depend on the choice of $Y$ in $C$.
\end{lemma}

\begin{remark}
If $X$ is a repetitive Delone set and $\Omega$ is its hull, then it is direct that $R(C_{Y,S})\leq M_{X}(S)$  for every $Y\in\Omega^0$ and $S>0$. Hence, in the linearly repetitive case we have $R(C_{Y,S})\leq LS$, where $L>1$ is the constant of linear repetitivity. Moreover, an estimation of $r(C)$ in terms of $L$ (known as a repulsion property) also exists (see \cite{Lenz}). For reference, these estimations are given below. 
\end{remark}

\begin{lemma}
Let $X$ be a linearly repetitive Delone set with constant $L>1$. Then, for every cylinder $C_{Y,S}$ with $Y\in\Omega^0$ and $S>0$ we have
\begin{equation}\label{eq:retbounds}
\frac{S}{2(L+1)}\le r(C_{Y,S})< R(C_{Y,S}) \le LS.
\end{equation}
\end{lemma}

\subsection{Solenoids, boxes and transverse measures.}
\label{sec:solenoids}
In this section, we recall some definitions and results of \cite{BBG,BG} that will be used throughout the paper.
Let $(\Omega,\Gamma)$ be an aperiodic minimal Delone system. The hull $\Omega$ is locally homeomorphic to the product of a Cantor set and $\RR^d$
(see \cite{AP,SW}). Moreover, there exists an open cover $\{U_i\}_{i=1}^n$ of $\Omega$ such that for each $i\in\{1,\ldots,n\}$, there are $Y_i\in \Omega$, $S_i>0$ and open sets $D_i\subseteq \RR^d$ such that $U_i = C_{Y_i,S_i}[D_i]$ and the map $h_i:D_i\times C_i \rightarrow U_i$ defined by $h_i(t,Z) = Z - t$ is a homeomorphism. Furthermore, there are vectors $v_{i,j}\in\RR^d$ (depending \emph{only} on $i$ and $j$) such that the transition maps $h_i^{-1}\circ h_j$ satisfy  
\begin{equation}
\label{transitionmaps}
 h_i^{-1} \circ  h_j(t,Z) = (t - v_{i,j}, Z - v_{i,j})
\end{equation}  
at all points $(t,Z)$ where the composition is defined. Following \cite{BG}, we call such a cover a \emph{$\RR^d$-solenoid's atlas}. 
It induces, among others structures, a laminated structure as follows. First, \emph{slices} are defined as sets of the form $h_i(D_i \times \{Z\})$. Equation \eqref{transitionmaps} implies that  slices are mapped onto slices. Thus, the  \emph{leaves} of $\Omega$ are defined as the smallest connected subsets that contain all the slices they intersect. 
It is not difficult to check, using \eqref{transitionmaps}, that the leaves coincide with the orbits of $\Omega$.

A \defi{box} in $\Omega$ is a set of the form $B :=C[D]$ where $C$ is a local transversal in $\Omega$, and $D\subseteq\RR^d$ is an open set such that the map from $D\times C$ to $B$ given by $(x,Y)\mapsto Y - x$ is a homeomorphism. This is true, for instance, if $D\subseteq B_{r(C)}(0)$ (\emph{cf.} \eqref{rX}). 

A Borel measure $\mu$ on $\Omega$ is \emph{translation invariant} if $\mu(B - v) = \mu(B)$ for every Borel set $B$ and $v\in\RR^d$. Let $C$ be a local transversal and $0<r<r(C)$. Each translation invariant measure $\mu$ induces a measure $\nu$ on $C$ (see \cite{Ghys} for the general construction and e.g. \cite{CFS} for the analog construction for flows): given a Borel subset $V$ of $C$, its \emph{transverse measure} is defined by
\begin{equation*}\nu(V) = \frac{\mu ( V[B_{r}(0)])}{\vol{B_{r}(0)}}.
\end{equation*}
This gives a measure on each $C$, which does not depend on $r$. The collection of all measures defined in this way is called the \emph{transverse invariant measure} induced by $\mu$. It is invariant in the sense that if $V$ 
is a Borel subset of $C$ and $x\in\RR^d$ is such that $V-x$ is a Borel subset of another local transversal $C'$, then $\nu(V-x) = \nu(V)$. Conversely, 
the measure $\mu$ of any box $B$ written as $C[D]$ may be computed by the equation 
\begin{equation*}
 \mu(C[D])=\leb{D}\times \nu(C).
\end{equation*}

\section{Tower systems}
\label{sec:boxes}
Let $\Omega$ be the hull of an aperiodic repetitive Delone set $X$. In this section we review the concepts
of box decompositions and tower systems introduced in \cite{BBG,BG} and prove our main result.  

\subsection{Box decompositions and derived tilings.}

A \defi{box decomposition} is a finite and pairwise-disjoint collection of boxes $\B = \{B_1,\ldots,B_{t}\}$ in $\Omega$ such that the closures of the boxes in $\B$ cover the hull. 
For simplicity, we always write $B_i=C_i[D_i]$, where $C_i$ and  $D_i$ are fixed and $C_i$ is contained in $B_i$. In particular, the set $D_i$ contains $0$. We  refer to $C_i$ as the \emph{base} of $B_i$. In this way, we call the union of all $C_i$ the \defi{base} of $\B$. The reasoning for fixing a local transversal in each $B_i$  comes from the fact that box decompositions can be constructed in a  canonical way starting from the set $\R_C(Y)$ of return vectors to a given local transversal $C$ (see details in Section \ref{proof:gambaudo}).

An alternative way of understanding a box decomposition is given by a family of tilings, known as \emph{derived tilings}, 
which are constructed by intersecting the box decomposition with the orbit of each Delone set in the hull. First, we 
recall basic definitions about tilings.
A \emph{tile} $T$ in $\RR^d$ is a compact set that is the closure of its interior (not necessarily connected). A \emph{tiling} $\T$ of $\RR^d$ is a countable collection of 
tiles that cover $\RR^d$ and have pairwise disjoint interiors. Tiles can be \emph{decorated}: they may have a color and/or be punctured at an interior point. Formally, this means that decorated tiles are tuples $(T,i,x)$, where $T$ is a tile,  $i$ lies in a finite set of \emph{colors}, and $x$ belongs to the interior of $T$. Two tiles have the same type if they differ by a translation. If the tiles are punctured, then the translation must also send one puncture to the other, and when they are colored, they must have the same color. 

To construct a derived tiling, the idea is to read the intersection of the boxes in the box decomposition with the orbit of a fixed Delone set in the hull. In the sequel will be convenient to make the following construction. Let $\{C_i\}_{i=1}^{t}$ be a collection of local transversals and $\{D_i\}_{i=1}^{t}$ be a collection of bounded open subsets of $\RR^d$ containing $0$. Define $\B = \{C_i[D_i]\}_{i=1}^{t}$ and observe that the sets in $\B$ are not necessarily boxes of $\Omega$. For each $Y\in\Omega$, define the (decorated) \emph{derived  collection} of $\B$ at $Y$ by 
 \begin{equation*}
\T_\B(Y) := \{ (\adh{D_i}+v,i,v) \mid  i\in\{1,\ldots,t\}, v\in \R_{C_i}(Y)\}.
\end{equation*}
The following lemma gives the relation between box decomposition and tilings.
\begin{lemma}
\label{lem:derivedtiling}
Let $\B = \{C_i[D_i]\}_{i=1}^{t}$, where the $C_i$'s are local transversals and the $D_i$'s are open bounded subsets of $\RR^d$ that contain $0$. Then, $\B$ is a box decomposition if and only if $\T_\B(Y)$ is a tiling of $\RR^d$ for every $Y\in\Omega$. In this case,  we call $\T_\B(Y)$ the \emph{derived tiling} of $\B$ at $Y$.
\end{lemma}
\begin{proof}
It is easy to see that if $\B$ is a box decomposition, then $\T_\B(Y)$ is a tiling for every $Y\in\Omega$. 
We now show the converse. For convenience, set $C=\cup_i C_i$. Fix  $Y\in \Omega$ and suppose there are $i,j\in\{1,\ldots,t\}$, $Y_1\in C_i,Y_2\in C_j$, $x_1\in D_i$ and $x_2\in D_j$ such that $Y = Y_1 - x_1 = Y_2 - x_2$. This implies that the tiles $\adh{D_i}-x_1$ and $\adh{D_j}-x_2$ of $\T_\B(Y)$ meet an interior point. Since $\T_\B(Y)$ is a tiling, these tile must coincide, and hence $i = j$ and $x_1 = x_2$. We conclude that maps $h_i:C_i\times D_i\rightarrow C_i[D_i]$ given by $(Y,t)\mapsto Y-t$ are one-to-one, and moreover $\B$ is pairwise disjoint. 

It rests us to prove that the map $h_i$ are homeomorphisms, and that the closures of the sets in $\B$ cover $\Omega$. Fix $i\in\{1,\ldots,t\}$. For, the map $h_i$ is the restriction of the translation action to $C_i\times D_i$, and therefore it is continuous. The continuity of the inverse  of $h_i$ follows from an standard argument involving the compactness of $C_i$ and the boundedness of $D_i$. Finally, given any Delone set $Y$ in $\Omega$, there is a tile $(\adh{D_i} + x , i, x)$ in $\T_\B(Y)$ that contains the origin, which clearly means that $Y$ belongs to the closure of $C_i[D_i]$. 
\end{proof}

\label{sec:towers}  

\subsection{Properly nested box decompositions.}  A box decomposition $\B'=\{C'_i[D'_i]\}_{i=1}^{t'}$ is \emph{zoomed out} of another box decomposition $\B = \{C_j[D_j]\}_{j=1}^{t}$ if the following properties are satisfied:
\begin{enumerate}[(Z.1)]
\item If $Y \in C'_i$ is such that $Y - x \in C_j - y$ for some $x\in \adh{D'_i}$ and $y\in \adh{D_j}$, then $C'_i - x \subseteq C_j - y$.
\item If $x \in \partial D'_i$, then there exist $j$ and $y\in\partial D_j$ such that $C'_i - x \subseteq C_j - y$.
\item For every box $B'$ in $\B'$, there is a box $B$ in $\B$ such that $B\cap B'\neq\emptyset$ and $\partial B\cap \partial B' = \emptyset$.
\end{enumerate}
For each $i\in\{1,\ldots,t'\}$ and $j\in\{1,\ldots,t\}$ define 
\begin{equation}
\label{def:Oij}
O_{i,j}=\{x\in D'_{i} \mid C'_{i}-x \subseteq C_{j}\}.
\end{equation}
\begin{enumerate}[(Z.4)]
\item For each $i\in\{1,\ldots,t'\}$ and $j\in\{1,\ldots,t\}$,
\begin{equation*}
 \adh{D'_{i}} = \bigcup_{j=1}^{t} \bigcup_{x\in O_{i,j}} \adh{D_{j}} + x,
\end{equation*}
where all the sets in the right-hand side of the equation have pairwise disjoint interiors. 
\end{enumerate}
Observe that in the case that $D_j$ is connected, then properties (Z.1) and (Z.2) imply (Z.4).

Since we are considering the $C'_i$'s and $C_j$'s as the bases of the boxes,  we ask the following additional property to be satisfied:
\begin{enumerate}[(Z.5)]
\item The base of $\B'$ is included in the base of $\B$, that is, $\cup_i C'_i \subseteq \cup_j C_j$.
\end{enumerate}

By (Z.4), we have that the tiling $\T_{\B'}(Y)$ is a super-tiling of $\T_\B(Y)$ in the sense that each tile $T$ in $\T_{\B'}(Y)$ can be decomposed into a finite set of tiles of $\T_\B(Y)$. By (Z.3), one of these tiles is  included in the interior of $T$.
\begin{lemma}
\label{lem:zoomed} 
For every $j\in\{1,\ldots,t\}$ we have
\begin{equation*}
C_{j} = \bigcup_{i=1}^{t'} \bigcup_{x\in O_{i,j}} C'_{i}-x.
\end{equation*}
\end{lemma}
\begin{proof}By the definition of $O_{i,j}$ and (Z.1), it suffices us to show that every $Y\in C_j$ belongs to the interior of some box $C'_i[D'_i]$. Suppose not, then $Y \in C'_i - x$ with $x\in \partial D'_i$ for some $i$ since since $\B'$ is a box decomposition. Moreover, by (Z.2) we deduce that $Y$ must be in the boundary 
of some box $B_{j'}$ in $\B$, which gives a contradiction.
\end{proof}
\subsection{Tower systems.}
A \defi{tower system} is a sequence of box decompositions $\mathfrak{T}=(\B_n)_{n\in\NN}$ such that $\B_{n+1}$ is zoomed out of $\B_n$ for all $n\in\NN$.   The following theorem was proved in  \cite{BBG}.
\begin{theorem}
Every aperiodic minimal Delone system possesses a tower system.
\end{theorem}
 
Consider a decreasing sequence $\mathfrak{C}=(C_{n})_{n\in\NN}$ of local transversals with diameter going to $0$, and a tower system $\mathfrak{T}$. We will suppose that $\mathfrak{T}$ is \emph{adapted} to $\mathfrak{C}$, i.e., that for all $n\in\NN$ we have $\B_n=\{C_{n,i}[D_{n,i}]\}_{i=1}^{t_n}$ such that $C_{n} = \cup_i C_{n,i}$ and $t_n$ is a positive integer. For each $n\in\NN^*$ we define, as in \eqref{def:Oij}, 
\begin{equation}
\label{On}
O_{i,j}^{(n)}=\{x\in D_{n,i}\mid C_{n,i}-x \subseteq C_{n-1,j}\}
\end{equation}
and 
\begin{equation*}
m_{i,j}^{(n)}=\sharp O_{i,j}^{(n)}
\end{equation*}
for every $i\in\{1,\ldots,t_n\}$ and $j\in\{1,\ldots,t_{n-1}\}$. The \emph{transition matrix} (associated to $\mathfrak{T}$)   of level $n$ is  defined as the matrix $M_n=(m_{i,j}^{(n)})_{i,j}$, i.e., $M_n$ has size $t_n\times t_{n-1}$.

Suppose that $\mu$ is a translation invariant probability measure and $\nu$ is the induced transverse measure (\emph{cf.} Section \ref{sec:solenoids}). From (Z.4),  Lemma \ref{lem:zoomed} and the definition 
 of transverse invariant measures, we  get 
\begin{equation}
\label{eq:vol}
\leb{D_{n,i}} = \sum_{j=1}^{t_{n-1}} m_{i,j}^{(n)} \leb{D_{n-1,j}}
\end{equation}
and 
\begin{equation}
\label{eq:tran}
\nu(C_{n-1,j})=\sum_{i=1}^{t_n}\nu(C_{n,i})m_{i,j}^{(n)}. 
\end{equation}
Fix $n\in\NN$.   From the relation $\mu(C_{n,i}[D_{n,i}]) = \vol{D_{n,i}}\nu(C_{n,i})$ and 
the fact that $\B_n$ is a box decomposition, it follows that
\begin{equation}\label{eq:meas}
\sum_{j=1}^{t_n}\leb{D_{n,j}}\nu(C_{n,j})=1 .
\end{equation}

Given a box decomposition $\B=\{C_{i}[D_{i}]\}_{i=1}^{t}$ , define   
its external and internal radius by 
\begin{align*}
R_\text{ext}(\B) &= \max_{i\in\{1,\ldots,t\}} \inf\{R>0 : B_R(0)\supseteq D_{i}\} ;\\
r_\text{int}(\B) &= \min_{i\in\{1,\ldots,t\}}\sup\{r>0 : B_r(0)\subseteq D_{i}\},
\end{align*}
respectively. Define also $\operatorname{rec}(\B) = \max_{i\in\{1,\ldots,t\}}\operatorname{rec}(C_{i})$.

\begin{thm}
\label{thm:towersLR}
Let $X$ be an aperiodic linearly repetitive Delone set with constant $L>1$ and $0\in X$. Given $K\geq 6L(L+1)^2$ and $s_0>0$, set $s_n = K^ns_0$ for all $n\in\NN$ and let $C_n := C_{X,s_n}$ for all $n\in\NN$. Then, there exists a tower system $\mathfrak{T}$ of $\Omega$ adapted to $(C_n)_{n\in\NN}$ 
that satisfies the following additional properties:
\begin{enumerate}[(i)]
\item for every $n\ge 0$,  $C_{n+1}\subseteq C_{n,1}$;
\item there exist constants 
\[K_1:=\frac{1}{2(L+1)} - \frac{L}{K-1}\quad \text{and}\quad K_2:=\frac{LK}{K-1},\]
which satisfy
$0<K_1<1<K_2$, such that for every $n\in\NN$ we have
\begin{equation}
\label{eqK1K2}
 K_1s_n\le r_\text{int}(\B_n) < R_\text{ext}(\B_n)\le K_2s_n;
\end{equation}
\item for every $n\in\NN$,
\begin{equation}
\label{eqRrec}
\operatorname{rec}(\B_n)\leq (2L+1) s_n.
\end{equation} 
\end{enumerate}
\end{thm}
The proof is deferred to Section \ref{proof:gambaudo}.
\subsection{Tower systems with uniformly bounded transition matrices.}
The following lemma allows to estimate the coefficients of the transition matrices.
\begin{lemma}
\label{prop:matrices}
Let $X$ be an aperiodic repetitive Delone set and $\mathfrak{C}=(C_n)_{n\in\NN}$ be a decreasing sequence of clopen subsets of 
$\Omega^0$ with $\diam{C_n}\rightarrow 0$ as $n\rightarrow+\infty$. Suppose that $\mathfrak{T}$ is a tower system
adapted to $\mathfrak{C}$. Then the following assertions hold:
\begin{enumerate}[(i)]
\item If  $M_X(\operatorname{rec}(\B_n))\le r_\text{int}(\B_{n+1})$ for every $n\in\NN$, then the coefficients of the transition matrices $M_n$ are strictly positive.
\item  If
$A:=\sup_{n\in\NN} ({R_\text{ext}(\B_{n+1})}/{r_\text{int}(\B_n)})$ is finite, then $\left\|M_n\right\|_\infty\leq A^d$ for all $n\in\NN^*$, and in particular, the set $\{M_n\}_{n\in\NN}$ is  finite. Here $\norm{M_n}_\infty$ denotes the maximum absolute row sum of $M_n$.
\end{enumerate}
\end{lemma}
\begin{proof}
Fix $n\in\NN$, $i\in\{1,\ldots,t_{n+1}\}$ and $Y\in C_{n+1,i}$. First, we prove (i). By hypothesis, all the $\operatorname{rec}(\B_n)$-patches of $Y$ occur in $D_{n+1,i}$. Since every $C_{n,j}$ is determined by a finite number of $\operatorname{rec}(C_{n,i})$-patches and these patches occur in $D_{n+1,i}$, it follows that there are vectors $v$ in $D_{n+1,i}$ such that $Y - v$ belongs to $C_{n,j}$. Hence $m_{i,j}^{(n)}>0$. 
We now prove (ii). Fix $n\in\NN^*$. Since  $D_{n,i}$ is included in a ball of radius $R_\text{ext}(\B_{n})$ and each $D_{n-1,j}$ contains a ball of radius $r_\text{int}(\B_{n-1})$, we deduce  from \eqref{eq:vol} that 
\begin{equation}\label{eq:1} 
\sum_{j=1}^{t_{n-1}} m_{i,j}^{(n)}\leq \left(\frac{R_\text{ext}(\B_{n})}{r_\text{int}(\B_{n-1})}\right)^d. 
\end{equation}
Taking maximum on $i$ in \eqref{eq:1} yields $\norm{M_n}_\infty\leq A^d$. The finiteness of $\{M_n\}_{n\in\NN^*}$ now follows from the last inequality.
\end{proof}

Finally, we state and prove our main result.
\begin{thm}
\label{thm:towersLR2}
Let $X$ be an aperiodic linearly repetitive Delone set. Then, the tower system of $\Omega$ obtained in Theorem \ref{thm:towersLR} satisfies the following: 
\begin{enumerate}[(i)]
\item for every $n\in\NN^*$, the matrix $M_n$ has strictly positive coefficients;
\item the matrices $\{M_n\}_{n\in\NN^*}$ are uniformly bounded in size and norm.
\end{enumerate}
\end{thm}
\begin{proof}
Take the notations of Theorem \ref{thm:towersLR} for $\mathfrak{C}$ and $\mathfrak{T}$. It suffices us to prove that $\mathfrak{T}$ satisfies the conditions  (i) and (ii) of Lemma \ref{prop:matrices}. Indeed, 
by the definition of linearly repetitivity we have $M_X(\operatorname{rec}(\B_n)) \le L \operatorname{rec}(\B_n)$
for all $n\in\NN$. Combining this with \eqref{eqRrec}, the left-hand inequality of \eqref{eqK1K2} and the definition of $s_n$ we get
\[M_X(\operatorname{rec}(\B_n)) \leq \frac{L(2L+1)}{KK_1} r_\text{int}(\B_{n+1}).\]
Since $K\geq 6L(L+1)^2$, it follows that $L(2L+1)\leq K_1K$ and the  condition (i) in Lemma \ref{prop:matrices} is satisfied. To check that the condition in (ii) is also satisfied, use \eqref{eqK1K2} twice (with $n$ and $n+1$) and then replace ${s_{n+1}}=K{s_{n}}$ into the result to obtain
\begin{equation*}
\frac{R_\text{ext}(\B_{n+1})}{r_\text{int}(\B_n)}\leq K\frac{K_2}{K_1},
\end{equation*}
from which it follows that  ${R_\text{ext}(\B_{n+1})}/{r_\text{int}(\B_n)}$ is uniformly bounded in $n\in\NN$. 
\end{proof} 
\begin{cor}
\label{cor:unique_erg}
Let $X$ be an aperiodic linearly repetitive Delone set and $\Omega$ its hull. Then the system $(\Omega,\Gamma)$ is uniquely ergodic. 
\end{cor}
The proof is standard and can be found e.g. in  \cite{BG,BBG}. It is also important (for the remainder of this paper) to remark that this produces an independent proof to the original one by Lagarias and Pleasants in \cite{LP}.

\section{The Bellissard-Benedetti-Gambaudo's construction in the linearly repetitive case}
\label{proof:gambaudo}
In this section we give the proof of Theorem \ref{thm:towersLR} by adapting the construction of \cite{BBG}. First, we need to recall some 
basic facts about Voronoi tilings. Given a Delone set $Y$ and a point $y\in Y$, the \defi{Voronoi cell} $\V_y(Y)$ of $y$ in $Y$ is defined by
\begin{equation*}
\V_y(Y)=\{z\in \RR^d \mid \forall y'\in Y ,||z-y|| \le
||z-y'|| \}.
\end{equation*}
It is standard that Voronoi cells are closed convex polyhedron in $\RR^d$, and they form  the so-called \emph{Voronoi tiling}, which we denote by $\T_Y$. The next lemma gathers well-known facts about Voronoi cells that will be needed in the proof (see \cite[Proposition 5.2 and Corollary 5.2]{Senechal}).
\begin{lemma}\label{lem:Vcell}
Let $Y$ be  a Delone set that is $R$-dense for $R>0$ and $r$-discrete for $r>0$. Then, for all $y\in Y$, the Voronoi cell $\V_y(Y)$ is included in $\adh{B}_R(y)$ and contains the ball $\adh{B}_{r}(y)$.
Moreover, the cell $\V_y(Y)$ is determined by the $2R$-patch of $Y$ centered at $y$, i.e., if $Y\wedge \adh{B}_{2R}(y)$ is equivalent to $Y\wedge \adh{B}_{2R}(y')$, 
then  
\[
\V_y(Y) - y  = \V_{y'}(Y)-y'.
\]
\end{lemma}
The fact that Voronoi cells are locally determined allows us to use them to construct box decompositions with any given clopen subset $C$ as base, as we show in the following result.  
\begin{prop}
\label{lem:VB}
Let $C$ be a clopen susbset of $\Omega^0$ and take $k\ge 2R(C)+\operatorname{rec}(C)$. Then, there is a box decomposition $\B(C,k)$ of $\Omega$ with $C$ as its base and  $\operatorname{rec}(\B(C,k))\leq k$, $r_\text{int}(\B(C,k))=r(C)$ and $R_\text{ext}(\B(C,k))=R(C)$.
\end{prop}
\begin{proof}
Consider
\[\A_{k,C} = \{ Y\wedge \adh{B}_k(0) \mid  Y\in C \}.\]

By minimality, each patch in $\A_{k,C}$ occurs  in $X$ and since $X$ has only finitely many $k$-patches up to translation, it follows that $\A_{k,C}$ is finite, say $\A_{k,C}=\{\pp_1,\ldots,\pp_t\}$, where $t\in\NN$ and the $\pp_i$'s are all different. For each $i\in\{1,\ldots,t\}$, we define
\begin{equation*}
C_i = \{Y\in \Omega \mid Y\wedge \adh{B}_k(0) = \pp_i\}
\end{equation*}
and
\[B_i = \{ Y - y \mid Y\in C_i, y\in\operatorname{int}(\V_{0}(\R_C(Y)))\}.\]
Let us  show that the $B_i$'s are boxes. On the one hand, the inequality $k\geq \operatorname{rec}(C)$ implies that all the $C_{i}$'s  are included in $C$. Hence, they partition $C$ and it follows that each $C_i$ is a clopen subset of $C$. On the other hand, $\R_C(Y)$ is $R(C)$-dense for every $Y\in C$. Hence,
Lemma \ref{lem:Vcell} implies that the Voronoi cell $\V_{0}(\R_C(Y))$ is determined by the patch $\pp(Y):=\R_C(Y)\wedge \adh{B}_{2R(C)}(0)$. Since $\pp(Y)$ is determined by the patch 
$Y\wedge \adh{B}_{2R(C)+\operatorname{rec}(C)}(0)$ and  $k\geq 2R(C)+\operatorname{rec}(C)$, it follows that  $\V_{0}(\R_C(Y))$ is  the same for all $Y\in C_i$ and therefore $B_i=C_i[D_i]$, where $D_i:= \operatorname{int}(\V_{0}(\R_C(Y)))$.
To conclude we define 
\[\B(C,k)=\{B_1,\ldots, B_t\}\]
and observe that the undecorated version of $\T_{\B(C,k)}(Y)$ is the Voronoi tiling of $\R_C(Y)$, which implies by Lemma \ref{lem:derivedtiling} that $\B(C,k)$ is a box decomposition. The clopen $C$ can be chosen as base of $\B(C,k)$ by construction and the last three relations in the statement of the proposition are direct consequences of Lemma \ref{lem:Vcell}.
\end{proof}

\begin{figure}
 \begin{center}
 \includegraphics[scale=0.5]{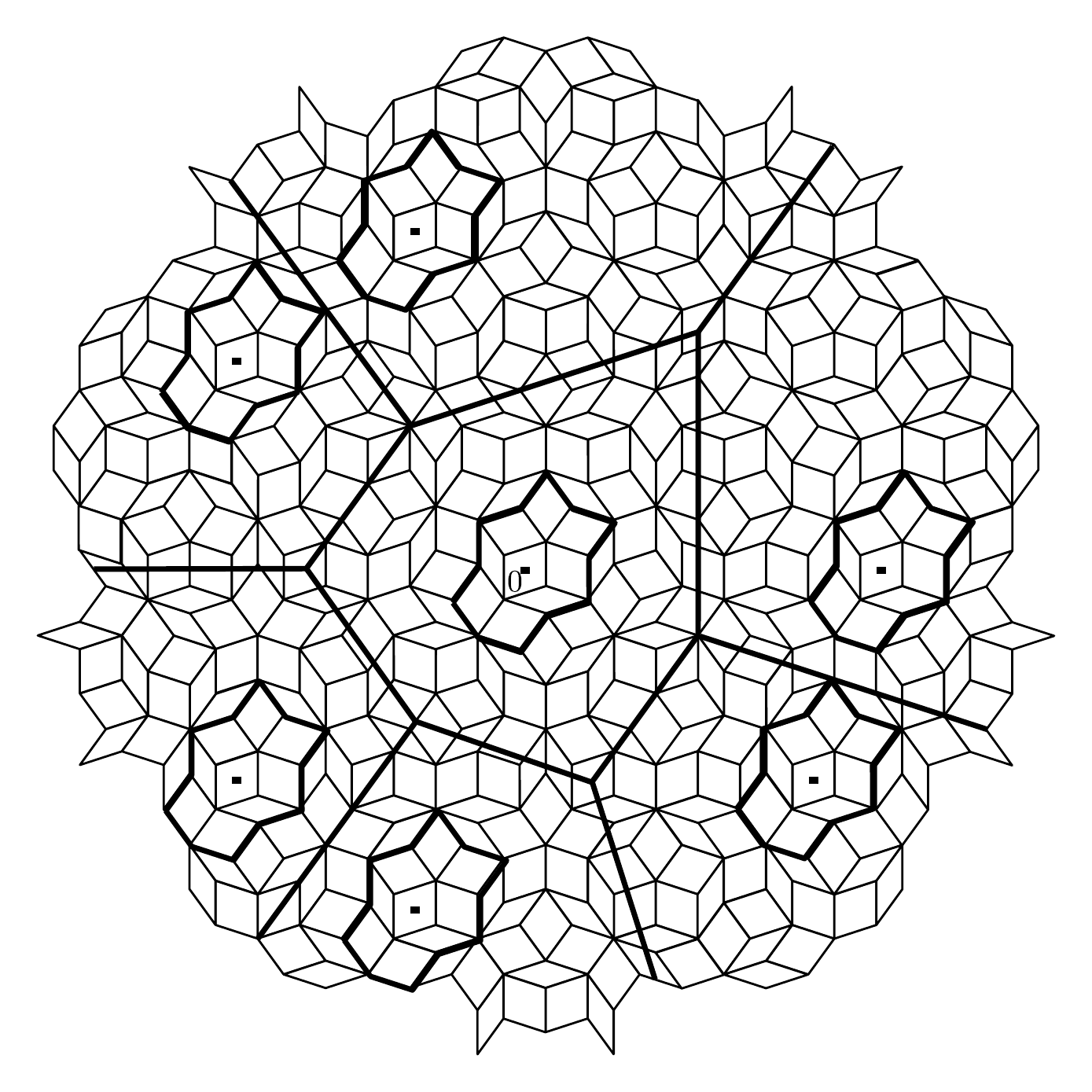}
\end{center}
\caption[Voronoi tile]{Voronoi tiling obtained from the Delone set of centers of translated copies of a patch.}\label{fig:PenVoronoi}
\end{figure}
The next step in the proof consists in, given a box decomposition $\B$ based at $C$, constructing a box decomposition zoomed out of $\B$. Before stating this precisely, we sketch the idea of the proof. Consider $\T$ the derived tiling of $\B$ at some Delone set $Y$ of $\Omega$. Let $C'$ be a clopen subset of $C$ and consider the Voronoi tiling $\V$ of $\R_{C'}(Y)$. Without considering colors, the tiling $\V$ is the derived tiling of the box decomposition $\B(C',k)$ given by Proposition \ref{lem:VB}, where $k$ is chosen large enough. In general, $\B(C',k)$  is not zoomed out of $\B$ because (see e.g. Figure \ref{fig:PenVoronoi}), the tiles of $\V$ can not decomposed into tiles of $\T$.

Then, we modify the tiles of $\V$ around their boundaries (without modifying their punctures so we do not modify the bases of the boxes) in such a way that, after the modification, the tiles can be decomposed into tiles of $\T$ (see Figure \ref{fig:PenSurgery}). The idea is to replace each tile $V$ in $\V$ by the union of the tiles in $\T$ with punctures in $V$. In the case that there is a tile of $\T$ with its puncture in the boundary of a tile of $\V$, then the new collection will not form a tiling (since this tile would belong to two different tiles of $\V$, which will, consequently, overlap). To solve this problem, a choice has to be made so each point of $\RR^d$ belongs to a unique tile of $\V$.   
\begin{figure}
\begin{center}
\includegraphics[scale=0.5]{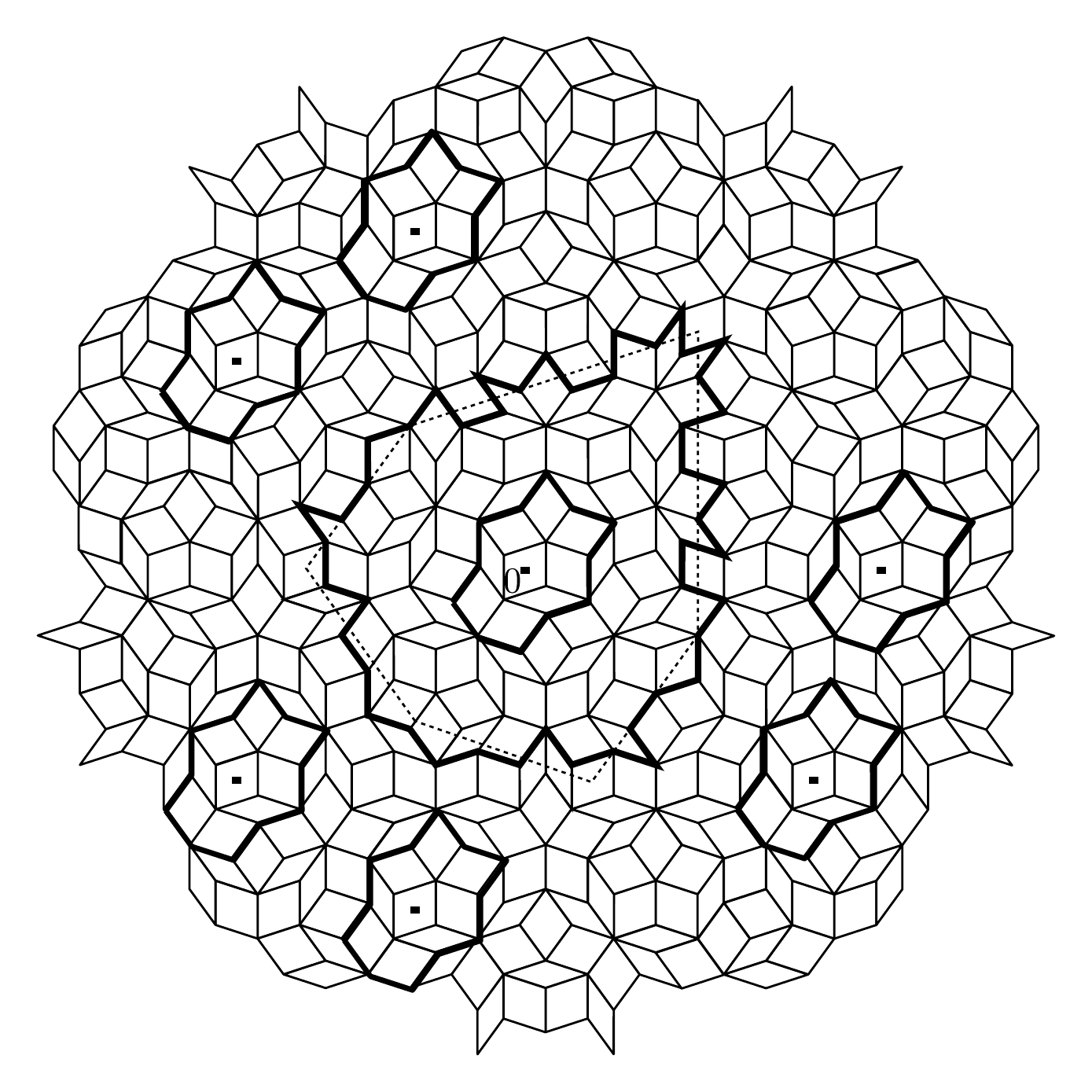}
\end{center}
\caption{Patch of the Penrose tiling constructed from the tiles intersecting the Voronoi tile.}\label{fig:PenSurgery}
\end{figure}
One way of achieving this is the following (see \cite{GMPS}). Given any subset $D$ of $\RR^d$, we define 
\begin{equation}
\label{estrella}
 D^*=\{ p \in\RR^d \mid p+(\varepsilon,\varepsilon^2,\ldots,\varepsilon^d)\in D \text{ for all sufficiently small $\varepsilon>0$}\}. 
\end{equation}
We have that $(D+v)^* = D^* + v$ for every $v\in\RR^d$, and if $\S$ is any tiling of $\RR^d$, then $\{S^*\mid S\in\S\}$ is a partition. Thus, we replace each tile $V$ of $\V$ by the tiles whose puncture belong 
to $V^*$. 
 
\begin{lemma}\label{lem:zooming}
 Let $\B$ be a box decomposition of $\Omega$ based at a clopen subset $C$ of $\Omega^0$ and $C'$ be a clopen subset of $C$
satisfying $r(C')\ge 2R_\text{ext}(\B)$. Then for each $k'\geq \max\{2R(C')+\operatorname{rec}(C'), R(C')+2R_\text{ext}(\B)+\operatorname{rec}(\B)\}$ there is a box decomposition 
$\B'$ zoomed out of $\B$ and based at $C'$ satisfying 
\begin{align}
\label{RrecBpr} \operatorname{rec}(\B')&\leq k',\\
\label{rBpr} r_\text{int}(\B')&\geq  r(C')-R_\text{ext}(\B),\\
\label{RBpr} R_\text{ext}(\B')&\leq  R(C') + R_\text{ext}(\B).
\end{align}
\end{lemma}
\begin{proof}
Let $\B = \{C_j[D_j]\}_{j=1}^{t}$ and $\B(C',k')=\{C'_i[E_i]\}_{i=1}^{t'}$. The proof consists in modifying the $E_i$'s  in $\B(C',k')$ to obtain a box decomposition $\B'$ that is zoomed out of $\B$. We proceed by steps and work with derived tilings.

\noindent \textbf{Step 1.} For each $Y\in \Omega$, we ``deform'' each tile of $\T_{\B(C',k')}(Y)$ into a tile that is the support of a patch of $\T_{\B}(Y)$. More precisely, for each tile $(\adh{E_i} + v,i,v)$ with $i\in\{1,\ldots,t'\}$ and $v\in \R_{C'_i}(Y)$ we define 
$D_{v,Y}'$ as
\begin{equation*}
D_{v,Y}'=\bigcup_{j=1}^{t}\bigcup_{w\in\R_{C_{j}}(Y)\cap ({E}_i^*+v)} \adh{D_j} + w,
\end{equation*}
where $E_i^*+v$ is defined by \eqref{estrella}. It is easy to see that $D_{v,Y}'$  is  the support of the tiles of $\T_\B(Y)$ whose punctures are in ${E}_i^*+v$. Since $\{E_i^*+v\mid i\in\{1,\ldots,t\}, v\in\R_{C'_i}(Y)\}$
is a partition of $\RR^d$, it follows that $\T'(Y) = \{(D_{v,Y}',i,v)\mid i\in\{1,\ldots,t\}, v\in\R_{C'_i}(Y)\}$ is a decorated tiling (the tiles are not necessarily connected), which we view as a deformation of $\T_{\B(C',k')}(Y)$.

\noindent \textbf{Step 2.} Fix $i\in\{1,\ldots,t'\}$ and denote $S = R(C')+\operatorname{rec}(\B)$. We show that $D_{w,Z}' = D_{v,Y}' - v + w$ for every $Y, Z\in\Omega$, $v\in \R_{C_i'}(Y)$ and $w\in \R_{C_i'}(Z)$. Indeed,  since ${E}_i^*$ included in $\adh{E_i}$, from Lemma \ref{lem:Vcell} we get that $E_i^*\subseteq \adh{B}_{R(C')}(0)$. It follows that the 
set \begin{equation*}
\R_{C_j}(Y)\cap (E_i+v)^*
\end{equation*}
is determined by $Y\wedge \adh{B}_S(v)$. From the proof of Proposition \ref{lem:VB}, there exists 
$Y_i\in \Omega$ such that $C_i' = C_{Y_i,k'}$. Hence, $Y \wedge \adh{B}_{k'}(v)$ is equivalent to $Z \wedge \adh{B}_{k'}(w)$. Since $k'\geq S$, it follows that $D_{w,Z}' = D_{v,Y}' - v + w$.

For each $i\in\{1,\ldots,t'\}$, we set $D_i' = D_{v,Y}' - v$ and $\B'=\{C_i'[D_i']\}_{i=1}^{t'}$, where $Y\in \Omega$ and $v\in\R_{C_i'}(Y)$ are arbitrary. It is clear that $\T_{\B'}(Y)=\T'(Y)$, which by Lemma \ref{lem:derivedtiling} implies that $\B'$ is a box decomposition.

\noindent \textbf{Step 3.} We check \eqref{RrecBpr}, \eqref{rBpr} and \eqref{RBpr}. Indeed, \eqref{RrecBpr} follows directly from Proposition \ref{lem:VB}. Also from Proposition 4.2, we have that $B_{r(C')}(0)\subseteq E_i\subseteq B_{R(C')}(0)$ for every $i\in\{1,\ldots,t'\}$. Since $D_j\subseteq B_{R_\text{ext}(\B)}(0)$ for all $j\in\{1,\ldots,t\}$, from Step 2 we deduce that each $D_i'$ contains the ball of radius $r(C')-R_\text{ext}(\B)$ around $0$ and is contained in the ball of radius $R(C')+R_\text{ext}(\B)$ around $0$, which yields \eqref{rBpr} and \eqref{RBpr}. 

\noindent \textbf{Step 4.} We check (Z.1). Indeed, suppose that  $Y\in C_i' - y$ belongs to $C_j - x$ for some $j\in\{1,\ldots,t\}$, $y\in \adh{D_i'}$ and $x\in \adh{D_j}$. Let $Z\in C_i'- y$. We need to show that $Z\in C_j - x$, which is equivalent to $Z\wedge \adh{B}_{\operatorname{rec}(\B)}(-x) = Y\wedge \adh{B}_{\operatorname{rec}(\B)}(-x)$. From \eqref{RBpr} we get $\norm{y}\leq R(C')+R_\text{ext}(\B)$.
Hence we have $k'-\norm{x-y}\geq\operatorname{rec}(\B)$. Since by hypothesis $Y\wedge\adh{B}_{k'}(-y) = Z\wedge\adh{B}_{k'}(-y)$, it follows that  $\adh{B}_{\operatorname{rec}(\B)}(-x)\subseteq \adh{B}_{k'}(-y)$ and the proof of (Z.1) is done.

We check (Z.3). Let $Y\in C_i'$. From \eqref{rBpr} and $r(C')\geq 2R_\text{ext}(\B)$, we get $r_\text{int}(\B')>R_\text{ext}(\B)$. Since $C'\subseteq C$, there is $j\in\{1,\ldots,t\}$ such that $Y\in C_j$, which means that the tile of $\T'(Y)$ containing the origin contain the tile of $\T_\B(Y)$ containing the origin in its interior. Since $Y$ was arbitrary, this implies (Z.3). 

Finally, (Z.2), (Z.4) and (Z.5) are direct consequences of the construction.

\end{proof}

Now we are ready to give the proof of Theorem \ref{thm:towersLR}.
\begin{proof}[Proof of Theorem \ref{thm:towersLR}.]
For each $n\in\NN$ we define $k_n = 2R(C_n) + s_n$. We construct a tower system by induction on $n$ and use Lemma \ref{lem:zooming} as a key part in the induction step. The estimates given by Lemma \ref{lem:zooming} are then used to prove properties (i),(ii) and (iii).

For the basis step, we set $\B_0 = \B(C_0,k_0)$. Since $s_0\geq \operatorname{rec}(C_0)$ by definition of $C_0$, $k_0\geq 2R(C_{0})+\operatorname{rec}(C_{0})$. Thus Proposition \ref{lem:VB}
ensures that $\B(C_0,k_0)$ is a box decomposition with $\operatorname{rec}(B_0)\leq k_0$ and $R_\text{ext}(\B_0) = R(C_0)$. By permuting the indices if necessary, we obtain that $X\in C_{0,1}$.

For the inductive step, we fix $n\in \NN^*$ and suppose that $\B_{n-1}$ is a box decomposition that satisfies:
\begin{align}
\label{RrecBn}
\operatorname{rec}(\B_{n-1})&\leq k_{n-1}\\
\label{RBn}
 R_\text{ext}(\B_{n-1})&\leq \frac{L}{K-1} s_{n}.
\end{align}
We need to show that $\B_{n-1}$ and $k_{n}$ satisfy the hypotheses of Lemma \ref{lem:zooming}:
\begin{align}
\label{hyp1}
k_{n}&\geq 2R(C_{n})+\operatorname{rec}(C_{n}),\\
\label{hyp2}
k_{n}&\geq R(C_{n})+2R_\text{ext}(\B_{n-1})+\operatorname{rec}(\B_{n-1}),\\
\label{hyp3}
r(C_{n})&\geq 2R_\text{ext}(\B_{n-1}).
\end{align}
The inequality \eqref{hyp1} is clear since $s_n\ge \operatorname{rec}(C_n)$ by the definition of $C_n$. To check \eqref{hyp2}, recall that $R(C_{n-1})\leq Ls_{n-1}$ by linearly repetitivity.  Replacing this inequality and $s_{n} = Ks_{n-1}$ in the definition of $k_n$ yields
\begin{equation}
\label{knSn1}
k_{n-1}\leq \frac{2L+1}{K} s_{n}.
\end{equation}
From  \eqref{knSn1}, \eqref{RrecBn} and \eqref{RBn} it is easy to deduce 
\begin{equation}
\label{knsn2}
2R_\text{ext}(\B_{n-1}) + \operatorname{rec}(\B_{n-1}) \leq \left(\frac{2L}{K-1}  + \frac{2L+1}{K}\right) s_{n}.
\end{equation}
An easy computation \TODO{WE NEED THAT $K\geq 4L+2$ for the following to be true} shows
that the right-hand  side of \eqref{knsn2} is smaller or equal than $s_n$ and thus \eqref{hyp2} follows from the definition
of $k_{n}$. Finally, \eqref{hyp3} follows from an easy computation involving the left-hand side of \eqref{eq:retbounds} and \eqref{RBn}.
\TODO{For this we need that $K\geq 4L(L+1)+1$.}

Applying  Lemma \ref{lem:zooming} we get a  box decomposition $\B_{n}=\{C_{n,i}[D_{n,i}]\}_{i=1}^{t_{n}}$ zoomed out of $\B_{n-1}$ that satisfies
\begin{align}
\label{RrecBn1}
\operatorname{rec}(\B_n)&\leq k_n\\
\label{rBn1}
r(C_n)-R_\text{ext}(\B_{n-1})&\leq r_\text{int}(\B_n)\\
\label{RBn1}
R_\text{ext}(\B_n)&\leq R(C_n) + R_\text{ext}(\B_{n-1}).
\end{align}
To finish the inductive step, it remains to show that
\begin{equation}
\label{RBn2}
 R_\text{ext}(\B_n)\leq \frac{L}{K-1}s_{n+1}.\end{equation}
This is proved easily by replacing \eqref{eq:retbounds} and \eqref{RBn} into \eqref{RBn1}.
Hence, applying the induction above we obtain a sequence of box decompositions $(\B_n)_{n\in\NN}$ such that $\B_{n}$ is zoomed out of $\B_{n-1}$ and satisfies \eqref{RrecBn1},\eqref{rBn1} and \eqref{RBn2}.

Finally we check properties (i),(ii) and (iii). After permuting indices we have that $X\in C_{n,1}$ for every $n\in\NN$. Since $s_{n+1} > k_{n}\geq \operatorname{rec}(\B_n)$, it follows that $C_{n+1}=C_{X,s_{n+1}}\subseteq C_{n,1}$ and (i) holds.

Replacing \eqref{eq:retbounds} and \eqref{RBn} into \eqref{rBn1} we obtain
\begin{equation}
\label{rrBn}
 r_\text{int}(\B_{n})\geq s_{n}\left(\frac{1}{2(L+1)} - \frac{L}{K-1}\right)\quad \text{ for all } n\in \NN.
\end{equation} 
Hence, property (ii) follows from \eqref{rrBn} and \eqref{RBn2}.

Finally, property (iii) follows directly from \eqref{knSn1} and \eqref{RrecBn1}.

\end{proof}

\section{Markov chain induced by a tower system}
\label{sec:markov}
Suppose that $\Omega$ is the hull of an aperiodic linearly repetitive Delone set with constant $L>1$. 
Take $K$, $\mathfrak{C}$ and $\mathfrak{T}=\{\B_n\}_{n\in\NN}$ as in Theorem  \ref{thm:towersLR}.
Let $\mu$ be the unique translation invariant probability measure on $\Omega$ and denote by $\nu$ the induced  transverse measure (cf. Section  \ref{sec:solenoids}).  

The tower system $\mathfrak{T}$ induces a random process $\beta = (\beta_n)_{n\in\NN}$ on $(\Omega,\mu)$ as follows. For every $Y \in\Omega$ and $n\in\NN$,  define $\b_{n}(Y)$ by 
\begin{equation*}
\b_n(Y) = i \quad\text{if and only if $Y$ belongs to $B_{n,i}$},
\end{equation*}
where $\B_n=\{B_{n,i}\}_{i=1}^{t_n}$. Observe that $\b_n$ is not defined at the boundaries of the boxes 
in $\B_n$. Since the boundaries have measure zero, it follows that $\beta$ is well-defined in a full-measure set. 
We also set 
\begin{equation*}
c_\mathfrak{T} =  1 - \left(\sup_{n\in\NN}\left\|M_n\right\|_1^{-1}\left\|M_{n+1}\right\|_1^{-1}\right),
\end{equation*}
where $M_n$ are the transition matrices of $\mathfrak{T}$ and $\norm{M_n}_1$ is the maximum absolute column sum of $M_n$, which are bounded in norm by Theorem \ref{thm:towersLR2}. The 
following proposition resumes the properties of $\beta$.
\begin{prop}
\label{lem:speed}
The process $\b$ is a non-stationary Markov chain with its transition probabilities given by 
\begin{equation}
\label{qntransmatr}
\mu(\beta_{n}= i \mid \beta_{n-1} = j) =\frac{\nu(C_{n,i})}{\nu(C_{n-1,j})}m_{i,j}^{(n)},
\end{equation}
Moreover, for every $n,m\in\NN$ with $n> m$ we have
\begin{equation}\label{eq:speed}
\max\limits_{\substack{1\le j \le t_{m}\\ 1\le i \le t_{n}}} \left|\mu(\beta_{n}= i \mid \beta_{m} = j)-\mu(\beta_{n}=i)\right|\leq c_\mathfrak{T}^{n-m}.
\end{equation}
\end{prop}
The proof will be divided in several lemmas. First, we introduce some notation. For each $n\in\NN^*$, we define the matrix $Q_n = (q_{j,i}^{(n)})$ with $j\in\{1,\ldots,t_{n-1}\}$ and $i\in\{1,\ldots,t_n\}$ by 
\begin{equation}
\label{def:qn}
q_{j,i}^{(n)} = \frac{\nu(C_{n,i})}{\nu(C_{n-1,j})}m_{i,j}^{(n)}.
\end{equation}
For $n>m$, we  define the products $Q(n,m) := Q_{m+1}\cdot \ldots \cdot Q_{n}$ and $P(n,m):= M_n\cdot\ldots\cdot M_{m+1}$, and denote its coefficients by $q_{j,i}^{(n,m)}$ and $p_{i,j}^{(n,m)}$, respectively. It is not difficult to check, using induction and \eqref{def:qn}, that  
\begin{equation}
\label{def:qnm}
q_{j,i}^{(n,m)} = \frac{\nu(C_{n,i})}{\nu(C_{m,j})}p_{i,j}^{(n,m)}
\end{equation}
for all $n,m\in\NN$ with $n>m$, $i\in\{1,\ldots,t_{n}\}$ and $j\in\{1,\ldots,t_{m}\}$. 
\begin{lemma}
\label{lemma:markov1}
Fix
$n,m\in\NN$ with $m<n$. Then for every sequence $(i_k)_{k=m}^{n}$ with  $i_k\in\{1,\ldots,t_k\}$  we have
\begin{equation}\label{eq:mark}
\mu(B_{n,i_n} \cap B_{n-1,i_{n-1}} \cdots \cap B_{m,i_m}) = (\Pi_{k=m}^{n-1} m_{i_{k+1},i_{k}}^{(k+1)})\leb{D_{m,i_m}}\nu(C_{n,i_n}).
\end{equation}
\end{lemma}
\begin{proof}From (Z.4) and \eqref{On} it is easy to deduce that  
\[ B_{m+1,i_{m+1}} \cap B_{m,i_{m}} = \bigcup_{x\in O_{i_{m+1},i_{m}}^{(m+1)}} C_{m+1,i_{m+1}}[D_{m,i_{m}}+x].\]
Using induction, it is possible to obtain 
\[ B_{n,i_n} \cap \ldots\cap B_{m+1,i_{m+1}} \cap B_{m,i_{m}} = \bigcup_{x_n\in O_{i_{n},i_{n-1}}^{(n)}}\ldots\bigcup_{x_{m+1}\in O_{i_{m+1},i_{m}}^{(m+1)}} C_{n,i_{n}}[D_{m,i_{m}}+\sum_{k=m+1}^n x_k].\]
Since the boxes in the right hand side of the last equation have disjoint interiors, using the invariance of $\mu$ we obtain \eqref{eq:mark}.
\end{proof}
For each $n\in\NN^*$,  define $c(Q_n):=1-\min_{i,j} q_{j,i}^{(n)}$. 

\begin{lemma}\label{lem:c}
If $c:=\sup_n(c(Q_n))<1$, then for every $n,m\in\NN$ with $n>m$ we have 
\begin{equation*}
\max_{i,j,s}\abs{q^{(n,m)}_{i,s}-q^{(n,m)}_{j,s}}\leq c^{n-m}.
\end{equation*}
\end{lemma}
\begin{proof}
The proof follows by applying \cite[Equations (4.6) and (4.7) p. 137-138]{Seneta}, which remain true in our setting.
\end{proof}
\begin{lemma}\label{lem:Qc}
For every $n\in\NN$ we have 
\begin{equation*}
\sup_n c(Q_n)\leq c_\mathfrak{T}.
\end{equation*}
\end{lemma}
\begin{proof}
Fix $n\in\NN^*$. First, we estimate $\nu(C_{n,k})/\nu(C_{n,j})$ for all $j$ and $k$ in $\{1,\ldots t_{n}\}$. 
From \eqref{eq:tran} we get 
\begin{equation}
\label{eq:12}
\frac{\nu(C_{n,k})}{\nu(C_{n,j})} = \sum_{i=1}^{t_{n+1}}\frac{\nu(C_{n+1,i})}{\nu(C_{n,j})}m^{(n+1)}_{ik}.
\end{equation}
Since $m_{i,j}^{(n+1)}\geq 1$ for all $i\in\{1,\ldots,t_{n+1}\}$, it follows from \eqref{eq:tran} that 
$\nu(C_{n,j})>\nu(C_{n+1,i})$ and hence from \eqref{eq:12} we get
\begin{equation}
\label{eq:123}
\frac{\nu(C_{n,k})}{\nu(C_{n,j})} 
\leq \sum_{i=1}^{t_{n+1}}m^{(n+1)}_{i,k}\leq \left\|M_{n+1}\right\|_1.
\end{equation}
Next, we estimate $\nu(C_{n-1,l})/\nu(C_{n,j})$ for all $j\in\{1,\ldots t_{n}\}$ and $l\in \{1,\ldots t_{n-1}\}$. Plugging  in \eqref{eq:tran} and \eqref{eq:123} we obtain
\begin{equation}\label{13}
\frac{\nu(C_{n-1,l})}{\nu(C_{n,j})}=\sum_{k=1}^{t_{n}}\frac{\nu(C_{n,k})}{\nu(C_{n,j})}m^{(n)}_{kl}\le \left\|M_n\right\|_1\left\|M_{n+1}\right\|_1.
\end{equation}
Finally, we estimate $q_{l,j}^{(n)}$. Plugging \eqref{13} in \eqref{def:qn} yields
\begin{equation}
q_{l,j}^{(n)}  \geq \left\|M_n\right\|_1^{-1}\left\|M_{n+1}\right\|_1^{-1} m_{j,l}^{(n)}\geq \left\|M_n\right\|_1^{-1}\left\|M_{n+1}\right\|_1^{-1},
\end{equation}
where we used that $m_{i,j}^{(n)}\geq 1$, and the conclusion now follows.
\end{proof}  
\begin{proof}[Proof of Proposition \ref{lem:speed}.]
The fact that $\beta$ is a Markov chain with transition probabilities given by 
\eqref{qntransmatr} can be proved by a simple computation using Lemma \ref{lemma:markov1}. To check \eqref{eq:speed}, 
using $\sum_{i}\mu(B_{n,i})=1$ we may write 
\[
\left|\mu(\b_{n}=i|\b_{m}=j)-\mu(\b_{n}=i)\right|  \leq \max_{l,j} \left|\mu(\b_{n}=i|\b_{m}=j)-\mu(\b_{n}=i|\b_{m}=l)\right| 
\]
for all $n,m\in \NN$ with $n>m$. Since the coefficients of $M_n$ are all positive, $\norm{M_n}_1>1$ and hence $c_\mathfrak{T}<1$ 
and the conclusion now follows from Lemma \ref{lem:Qc} and Lemma \ref{lem:c}.
\end{proof}

\section{An alternative proof of a theorem of Lagarias and Pleasants}

\label{secLP}
Let $X$ be a linearly repetitive Delone set with constant $L>1$ and $\mathfrak{T}$ be the tower system 
given by Theorem \ref{thm:towersLR2}. By Corollary \ref{cor:unique_erg}, the hull $\Omega$ is uniquely ergodic, which means
that $X$ has uniform patch frequencies, i.e., each $S$-patch has a well-defined frequency $\operatorname{freq}(\pp)$. We denote by $\mu$  the unique translation invariant probability measure and by $\nu$ its induced transverse measure.

We say that $U\subseteq \RR^d$ is a $d$-cube of side $N$ if $U = [0,N]^d + x$ for some $x\in\RR^d$. Given a $d$-cube $U\in\RR^d$ and 
an $S$-patch $\pp$ of $X$ with $S>0$, we estimate the \emph{deviation} of $\pp$ in $U$, which is defined as 
\[\operatorname{dev}_{\pp}(U) = n_{\pp}(U) - \vol{U}\freq(\pp),\]
and obtain 
\begin{teo}
\label{teoLP2}
Let $X$ be a linearly repetitive Delone set, $\Omega$ be its hull and  $\mathfrak{T}$ be the tower system of Theorem \ref{thm:towersLR2}. Define
\begin{equation}\label{deltamixing}
\delta_{\mathfrak{T}} = - \log_K {c_\mathfrak{T}},
\end{equation} 
where $\log_K$ is the logarithm in base $K$. Then, for every $S>0$ and every $S$-patch $\pp$ in $X$ we have 
\[|\operatorname{dev}_\pp(U_N)| = O(N^{d-\delta_\mathfrak{T}})\]
for all $N\in\NN$, where $U_N$ is a $d$-cube of side $N$ and the $O$-constant depends only on $\Omega$ and $\pp$. 
\end{teo}
This produces an alternative proof of Theorem \ref{teo:LPintro} (\emph{cf.} \cite[Theorem 6.1]{LP}). The main difference between our result and Theorem \ref{teo:LPintro} comes from the fact that Theorem \ref{teoLP2} relates the growth-rate of the deviation to the mixing rate of the Markov chain associated with $\mathfrak{T}$ through \eqref{deltamixing}. The proof will be given at the end of the secion. 
In the remainder of this section, we denote  by $(\T_n=\T_n(X))_{n\in \NN}$ the sequence of derived tilings of $X$ associated to $\mathfrak{T}$. 

Now we introduce a decomposition argument to estimate the deviation 
of an $S$-patch on a $d$-cube $U$ of side $N$, but first we need some notation. For each $S>0$, 
define $n_0 = n_0(S)$ to be the smallest integer such that 
\[ Y,Z\in C_{n,i} - x \quad{\text{implies}\quad} Y\wedge \adh{B}_S(0) = Z\wedge \adh{B}_S(0)\]
for all $i\in\{1,\ldots,t_n\}$ and all $x\in D_{n,i}$. We check that $n_0$ is finite. Indeed, from the construction of the towers, it follows that if $Y$ and $Z$ 
are in $C_{n,i}-x$, then they coincide in a ball or radius 
\begin{equation}
\label{size}k_n-R_\text{ext}(\B_n)
\end{equation} around $0$. It suffices to show that \eqref{size} goes to infinity as $n$ goes to infinity. Indeed, from \eqref{RBn1} and \eqref{eqK1K2}, we get $R_\text{ext}(\B_n)\leq R(C_n) + s_{n}L/(K-1)$. From the definition of $k_n$ and since $K-1>L$ we conclude that \eqref{size} goes to infinity, and therefore $n_0$ is finite. 

From the definition of $n_0$, we check that if $n\geq n_0$, then the number of $S$-patches of a Delone set $Y$ that are equivalent to a given $S$-patch $\pp$  and have their centers inside $D_{n,i}$ is the same for all $Y\in C_{n,i}$. Therefore, we write 
this number as $n_\pp(D_{n,i})$ and let 
\[\operatorname{dev}_\pp(D_{n,i}):=n_\pp(D_{n,i})-\operatorname{freq}(\pp)\vol{D_{n,i}}.\] Next, we define  $n_1=n_1(U)$ to be the biggest integer $n$ such that there is a tile of $\T_n$ included in $U$. First, we prove
the following lemma. 
\begin{lemma}
\label{lem:keyconvergence}
For each $S$-patch $\pp$ of $X$, and for all $d$-cubes $U$ of side $N$ in $\RR^d$,
\[
|\operatorname{dev}_{\pp}(U)| = O\left(N^{d-1}\left(1+\sum_{n = n_0}^{n_1} s^{1-d}_{n+1} \max_{1\le i\le t_n}|\operatorname{dev}_{\pp}(D_{n,i})|\right)\right). 
\]
\end{lemma}
\begin{lemma}\label{cntborde}
There exists $M> 0$ such that for every $d$-cube $U$ of side $N$ and every  $n\in\NN$ smaller than $n_1+1$, 
the number of tiles of $\T_{n}$ whose supports intersect  $U$ but they are not included in $U$  is bounded above by 
$MN^{d-1}s_n^{1-d}$. 
\end{lemma}
\begin{proof}
Fix a cube $U$ of side  $N\in\NN$ and $n\in\NN$, $n\leq n_1+1$. Denote by $\mathcal{A}$ the set of tiles of $\mathcal{T}_{n}$ whose supports intersect  $U$ but they are not included in $U$. 
From \eqref{eqK1K2} it follows that  each tile in $\mathcal{A}$ contains a ball of radius $K_1s_n$ and is included in a  ball of radius $K_2s_{n}$. 
Since tiles in $\Aa$ do not overlap, this implies that 
\begin{equation}
\label{lemaleb1}
K_1^d s_n^d \leb{B_1(0)} \#\mathcal{A} \leq \leb{(\partial U)^{+2K_2s_{n}}},
\end{equation}
where  $(\partial U)^{+2K_2s_{n}}=\{x\in \RR^d \mid \textrm{dist}(x,\partial U)\leq 2K_2s_{n}\}$.
It is easy to check that
\begin{equation}
\label{lemaleb2}
\leb{(\partial U)^{+2K_2s_{n}}}\le 2 d K_2s_{n}(2K_2s_{n}+N)^{d-1}.
\end{equation}
By definition of $n_1$, there is a tile $T$ in $\T_{n_1}$ that is included in $U$. The tile $T$ contains, by \eqref{eqK1K2},   a ball of radius $K_1 s_{n_1}$ and  hence  
\begin{equation}
 \label{cotan1}
K_1s_{n_1} \leq N.
\end{equation}
It follows that $s_{n}/N\leq K_1^{-1}$ since $(s_{n})_{n\in\NN}$ is increasing. Hence \eqref{lemaleb2} implies 
\begin{equation}
\label{lemaleb3}
\leb{(\partial U)^{+2K_2s_{n}}}\leq 2 d K_2 \left(2\frac{K_2}{K_1} + 1 \right)^{d-1} N^{d-1}s_{n}.
\end{equation}
The conclusion now follows from \eqref{lemaleb1} and \eqref{lemaleb3} with $M$ being defined by  
\[M = 2 d K_2(K_1^{d} \leb{B_1(0)})^{-1} K\left(2\frac{K_2}{K_1} + 1 \right)^{d-1}.\]
\end{proof}

In the following proof, we abuse the notation and identify the tiles of $\T_n$ (which are decorated) with their undecorated versions. In particular, 
if $T = (\adh{D_{n,i}},i,v)$ is a tile of $\T_n$, we write $n_\pp(T)$ and $\leb{T}$ for $n_\pp(D_{n,i})$ and $\leb{D_{n,i}}$.
\begin{proof}[Proof of Lemma \ref{lem:keyconvergence}.]
Fix an $S$-patch $\pp$ of $X$ and a $d$-cube $U$ of side $N$ in $\RR^d$.
The idea of the proof is to decompose $U$ into smaller pieces that are tiles of $\T_n$ for some 
$n\in\{n_0,\ldots,n_1\}$. Since the tiles of $\T_n$ are tiled by tiles of $\T_m$ for all $m\leq n$, we ask this decomposition to contain tiles as big as possible. 

More precisely, we define 
\begin{align*}
P_{n_1} &= \{T\in\T_{n_1}\mid T  \subseteq U\}\quad\text{and}\\
Q_{n_1} &= \bigcup_{T\in P_{n_1}} T.
\end{align*}
For $n\in\{n_0,\ldots,n_1-1\}$, $Q_n$ and $P_n$ are defined recursively as follows
\begin{align*}
P_{n} &= \left\{ T\in \T_{n} \mid  T \subseteq \adh{U\setminus Q_{n+1}}\right\},\\
Q_{n} &= \bigcup_{T\in P_{n}} T.
\end{align*}

Now we estimate, for every $n \in \{n_0,\ldots, n_1\}$, the cardinality of $P_n$. Fix  $n \in \{n_0,\ldots, n_1\}$.
By definition,  each tile in  $P_n$  lies inside a tile of $\T_{n+1}$ whose support intersects  $U$ but it is not included in $U$. 
By Lemma \ref{cntborde} there is at most $MN^{d-1}s_{n+1}^{1-d}$ of these tiles for some constant $M>0$ that does not depend on $n$. By Property (ii) in Theorem \ref{thm:towersLR2} there is a uniform bound  $\alpha>0$ for the number of tiles in $\T_n$ that form a tile in $\T_{n+1}$.  
Hence,
\begin{equation}
\label{cotaUn}
\# P_n \leq M \alpha^d N^{d-1}s_{n+1}^{1-d}\quad \text{ for all } n\in\{n_0,\ldots,n_1\}.
\end{equation}

Let $W = U\setminus\cup_{n=n_0}^{n_1} Q_n$. Since the $Q_n$'s do not overlap, we have
\begin{align}
\label{volu}
\leb{U} =& \leb{W}   + \sum_{n=n_0}^{n_1}\sum_{T\in P_n} \leb{T}.
\end{align}
Moreover, since for every $T\in \T_n$,  $X\cap \partial T = \emptyset$ we have that $n_{\pp}(T)=n_{\pp}(\interior{T})$. Hence,
\begin{align}
\label{occu}
n_{\pp}(U) =& n_{\pp}(W) + \sum_{n=n_0}^{n_1}\sum_{T\in P_n} n_{\pp}(T).
\end{align}
By definition of derived tiling, every tile in  $\T_n$ is a translation of some tile $T_{n,i}$ for some $i\in\{1,\ldots,t_n\}$. Since $n \geq n_0$ we obtain that $n_{\pp}(T)=n_{\pp}(T_{n,i})$. Hence
\begin{equation}
\label{errort}
|n_{\pp}(T) - \leb{T}\freq(\pp)|\leq \max_{i\in\{1,\ldots, t_n\}}\left|n_{\pp}(D_{n,i}) - \leb{D_{n,i}}\freq(\pp)\right|
\end{equation}
for every $t\in P_n$. Thus, from \eqref{occu}, \eqref{volu} and \eqref{errort} we obtain
\begin{equation}
\label{LP1}
\begin{split}
|n_{\pp}(U) - \freq(\pp)\leb{U}|\leq& \sum_{n=n_0}^{n_1} \# P_n \max_{i\in\{1,\ldots, t_n\}}|n_{\pp}(D_{n,i}) - \leb{D_{n,i}}\freq(\pp)|
\\&+|n_{\pp}(W)-\leb{W}\freq(\pp)|.
\end{split}
\end{equation}
By Lemma \ref{cntborde}, 
\begin{equation}
\label{errorborde}
\leb{W} \leq M N^{d-1}s_{n_0}^{1-d}  \max_{i\in\{1,\ldots ,t_{n_0}\}} \leb{D_{n_0,i}}. 
\end{equation}
Then, replacing \eqref{errorborde} and \eqref{cotaUn} into \eqref{LP1} gives the conclusion of the Lemma. 
\end{proof}

The next lemma allows us to estimate $\operatorname{dev}_\pp(D_{n,i})$ in terms of the coefficients
of the transition matrices of $\mathfrak{T}$ (\emph{cf.} Section \ref{sec:markov}). 
\begin{lemma}
\label{lem:decnp}
For all $n\geq n_0$ and $i\in\{1,\ldots,t_n\}$, 
\begin{equation*}
\operatorname{dev}_\pp(D_{n,i}) = \sum_{k=1}^{t_{n_0}}n_\pp(D_{n_0,k})(p_{ik}^{(n,n_0)}-\leb{D_{n,i}}\nu(C_{n_0,k})).
\end{equation*}
\end{lemma}
\begin{proof}
Suppose that $\pp = X \wedge B_S(x)$ with $x\in X$. It is well-known (see e.g.  \cite{LMS}) that $\operatorname{freq}(\pp)=\nu(C_\pp)$, where 
\[C_\pp = \{Y\in\Omega \mid Y\wedge B_S(0) = (X - x) \wedge B_S(0)\}.\]
On the one hand, from  \eqref{On}, the definition of $p_{ik}^{(n,n_0)}$ (\emph{cf.} Section \ref{sec:markov}) and the additivity of 
$n_\pp$ we deduce that 
\begin{equation}
\label{eq:decn}
n_\pp(D_{n,i}) = \sum_{k=1}^{t_{n_0}}n_\pp(D_{n_0,k})p_{ik}^{(n,n_0)}.
\end{equation}
On the other hand, since there is no occurrence of $\pp$ in the border of a box of $\B_{n_0}$, for every $Y\in C_\pp$, 
there are $k\in \{1,\ldots,t_{n_0}\}$, $Z\in C_{n_0,k}$ and $z\in D_{n_0,k}$ such that $Z - z \in C_\pp$. Moreover, 
the number of $z$ as above such that $Z - z \in C_\pp$ (with $k$ and $Z$ fixed) is exactly $n_\pp(D_{n_0,k})$. It follows
from the definition of $n_0$ that there are exactly $n_\pp(D_{n_0,k})$ copies of $C_{n_0,k}$ inside $C_\pp$ for all $k$ and hence 
\begin{equation}\label{eq:decC}
\nu(C_\pp) = \sum_{k=1}^{t_{n_0}} n_\pp(D_{n_0,k})\nu(C_{n_0,k}). 
\end{equation}
The conclusion now follows by   \eqref{eq:decC} multiplied by $\leb{D_{n,i}}$ from \eqref{eq:decn}.
\end{proof}

The last lemma before the proof of Theorem \ref{teoLP2} estimates the deviation of $p_{i,k}^{(n,n_0)}/\leb{D_{n,i}}$ with respect to its limit $\nu(C_{n_0,k})\vol{D_{n_0,k}}$; 
in terms of the mixing rate $c_\mathfrak{T}$ of the transition matrices. 
\begin{lemma}
\label{lem:ue}
For every $n > m\ge 0$ we have
 \[\max\limits_{\substack{1\le j \le t_{m}\\ 1\le i \le t_{n}}} 
\left|\frac{p_{i,j}^{(n,m)}}{\leb{D_{n,i}}}-\nu(C_{m,j})\right|= O(\nu(C_{m})c_\mathfrak{T}^{n-m}).\]
\end{lemma}
\begin{proof}
Let $M:=\sup_{n\in\NN^*}\norm{M_{n}}_\infty\norm{M_{n+1}}_\infty>1$. We prove that
for every $n\in\NN$ and $1\leq i \leq t_n$ we have 
\begin{equation}
\label{mea_box}
\leb{D_{n,i}}\nu(C_{n,i})\ge \frac{1}{M}.
\end{equation}
Indeed, by an argument analog to the one used in the proof of Lemma \ref{lem:Qc}, we get
\begin{equation}
\label{eq:proofue}
\frac{\leb{D_{n+1,k}}}{\leb{D_{n,i'}}} 
\leq \norm{M_{n+1}}_\infty \norm{M_{n}}_\infty
\end{equation}
for all $n>0$, $i'\in\{1,\ldots,t_n\}$ and $k\in\{1,\ldots,t_{n+1}\}$. Now, 
from \eqref{eq:tran} and $m_{i,j}^{(n)}\geq 1$ we deduce that $\nu(C_{n,i})\ge \sum_{j=1}^{t_{n+1}} 
\nu(C_{n+1,j})$. Hence, we have
\[\leb{D_{n,i}}\nu(C_{n,i})\geq \sum_{j=1}^{t_{n+1}}\leb{D_{n+1,j}} \nu(C_{n+1,j})\frac{\leb{D_{n,i}}}{\leb{D_{n+1,j}}}.\]
Thus, replacing  \eqref{eq:proofue} and \eqref{eq:meas} in the last inequality we get \eqref{mea_box}. Finally,  
the conclusion of the lemma follows from \eqref{def:qnm}, \eqref{eq:speed} and \eqref{mea_box}.
\end{proof}

\begin{proof}[Proof of Theorem \ref{teoLP2}]
We do the proof for the case $\delta_\mathfrak{T}<1$ (the other cases giving better estimates). By Lemma \ref{lem:keyconvergence}),
it suffices us to show that
\begin{equation}
\label{eq:propflp1}
\sum_{n = n_0}^{n_1} s^{1-d}_{n} \max_{1\le i\le t_n}|\operatorname{dev}_\pp(D_{n,i})| = O(N^{1-\delta_{\mathfrak{T}}}).
\end{equation}
Indeed, using  Lemma \ref{lem:decnp} and Lemma \ref{lem:ue} 
we obtain 
\begin{equation*}
|\operatorname{dev}_\pp(D_{n,i})| = \sum_{k=1}^{t_{n_0}}n_\pp(D_{n_0,k})O\left(c_\mathfrak{T}^{n-n_0} \leb{D_{n,i}} \right).
\end{equation*}
Recall that from \eqref{eqK1K2} we have $\leb{D_{n,i}} = O(s_n^d)$. Hence the left-hand side of \eqref{eq:propflp1}
can be estimated as 
\begin{equation*}
\sum_{n = n_0}^{n_1} s^{1-d}_{n} \max_{1\le i\le t_n}|\operatorname{dev}_\pp(D_{n,i})| = 
\sum_{n = n_0}^{n_1} s^{1-d}_{n} \sum_{k=1}^{t_{n_0}}n_\pp(D_{n_0,k})O\left(c_\mathfrak{T}^{n-n_0} s_n^d \right)
 = \sum_{n = n_0}^{n_1} O\left(s_n c_\mathfrak{T}^{n-n_0}\right).
\end{equation*}
Plugging $s_n = K^{n-n_0}s_{n_0}$ and using $K^{1-\delta_\mathfrak{T}}>1$ in the last equation we obtain
\begin{equation*}
\sum_{n = n_0}^{n_1} s^{1-d}_{n} \max_{1\le i\le t_n}|\operatorname{dev}_\pp(D_{n,i})| = 
  O\left( (K^{1-\delta_\mathfrak{T}})^{n_1-n_0+1}\right).
\end{equation*}
Finally, the conclusion follows from \eqref{cotan1} in the proof of Lemma \ref{cntborde}.
\end{proof}

\subsection*{Acknowledgements} Both authors would like to thank Jarek Kwapisz for several comments and suggestions. J.A.-P. would also like to thank to the Erwin Scrh\"odinger Institute in Vienna, where part of this work was done while J.A.-P. was a Junior Research Fellow. J.A.-P. acknowledges support from Fondecyt Postdoctoral Fund 3100097 and D.C. acknowledges support from Fondecyt Postdoctoral Fund 3100092. 

\providecommand{\bysame}{\leavevmode\hbox to3em{\hrulefill}\thinspace}
\providecommand{\MR}{\relax\ifhmode\unskip\space\fi MR }
\providecommand{\MRhref}[2]{%
  \href{http://www.ams.org/mathscinet-getitem?mr=#1}{#2}
}
\providecommand{\href}[2]{#2}

\end{document}